\documentclass[reqno,12pt]{amsart} %
\usepackage{amsmath,amstext, amsthm, amssymb,amsfonts}%,chicago}
\usepackage{fancybox}
\usepackage{color}
 \usepackage[top=1 in,bottom=1in, left=1 in, right= 1 in]{geometry}
\numberwithin{equation}{section}

\newtheorem{theorem}{Theorem}[section]
\newtheorem{proposition}[theorem]{Proposition}

\newtheorem{corollary}[theorem]{Corollary}
\theoremstyle{definition}

\newcount\mins \newcount\hours \hours=\time \mins=\time
\divide\hours by60 \multiply\hours by 60 \advance\mins by-\hours
\divide\hours by60

\def\now{ \ifnum\hours>11 \ifnum\hours>12 \advance\hours by
-12 \fi \number\hours:\ifnum\mins<10 0\fi \number\mins\ pm,\ \else
\ifnum\hours=0 \hours=12 \fi \number\hours:\ifnum\mins<10 0\fi
\number\mins\ am,\ \fi}

\newcommand{\V}{\mathbb{V}}

\newcommand{\bea}{\begin{equation}}
\newcommand{\eea}{\end{equation}}

\newcommand\utimes{\mathbin{\ooalign{$\cup$\cr%
   \hfil\raise0.42ex\hbox{$\scriptscriptstyle\times$}\hfil\cr}}}
\newcommand\bigutimes{\mathop{\ooalign{$\bigcup$\cr%
   \hfil\raise0.36ex\hbox{$\scriptscriptstyle\boldsymbol{\times}$}\hfil\cr}}}
\newcommand{\CSK}{Cauchy-Stieltjes Kernel (CSK)\ \renewcommand{\CSK}{CSK\ }}
\newcommand{\FEF}{Free Exponential  (FE)\ \renewcommand{\FEF}{FE\ }}

\title{ Cauchy-Stieltjes kernels families and free multiplicative convolution }

\author{Abdelhamid Hassairi* and Raouf Fakhfakh \ddag}

\address{ * Laboratory of Probability and Statistics,
Sfax University,  Sfax, Tunisia} \email{}

\address{\ddag  College of Science and Arts in Gurayat, Jouf University, Gurayat, Saudi Arabia \& Laboratory of Probability and Statistics,
Sfax University,  Sfax, Tunisia}
\email{fakhfakh.raouf@gmail.com}



\subjclass[2000]{ AMS Mathematics Subject Classification 2010}

\begin{document}

\begin{abstract}
In this paper, we determine the effect of the free multiplicative
convolution on the pseudo-variance function of a Cauchy-Stieltjes
kernel family. We then use the machinery of variance functions to
establish some limit theorems related to this type of convolution
and involving the free additive convolution and
the boolean additive convolution.
\end{abstract}

\maketitle

\section{Introduction}
It is well known that the most common probability distributions used
in statistics belong to natural exponential families whose
definition is based on the exponential kernel $\exp(\theta x)$. This
explains the importance of the theory of exponential families and
their variance functions both in probability theory and in
statistics. In the framework of free probability and in analogy with
the theory of natural exponential families, a theory of
Cauchy-Stieltjes kernel (CSK) families has been recently introduced,
it is based on the Cauchy-Stieltjes kernel $1/(1-\theta x)$. For
instance, Bryc in \cite{Bryc-06-08} has initiated the study of the
CSK families for compactly supported probability measures. He has in
particular shown that such families can be parameterized by the mean
$m$. With this parametrization, denoting $V(m)$ the variance of the
element with mean $m$, the function $m\mapsto V(m)$ called the
variance function and the mean $m_0$ of the generating measure $\nu$
uniquely determines the family and $\nu$. Bryc has also given the
effect on the variance function of the additive power of free
convolution $\boxplus$, more precisely, he has shown that for
$\alpha>0$,
\begin{equation}\label{Vbox+alpha}
V_{\nu^{\boxplus \alpha}}(m)=\alpha V_{\nu}(m/\alpha).
\end{equation}
 %He has also described the class of quadratic \CSK families. This class consists of the free Meixner distributions.
 In \cite{Bryc-Hassairi-09}, Bryc and Hassairi  have extended the results established in \cite{Bryc-06-08}
 to measures with unbounded support. They have provided a method to
determine the domain of means and introduced a notion of
pseudo-variance function which has no direct
 probabilistic interpretation but it has the properties of a variance function.
 These authors have also characterized the class of cubic CSK families with support bounded from one side, they
 have shown that this class is
 related to the quadratic class by a relation of reciprocity between
 tow Cauchy-Stieltjes kernel families expressed in terms of the $R$-transforms of the corresponding generating probability measures.
 A general description of polynomial variance
 function with arbitrary degree is given in \cite{Bryc-Raouf-Wojciech-18}. In particular, a complete
 description of the cubic compactly supported CSK families is given. Concerning the
 effect of the boolean additive convolution power on the variance function, it is shown in
 \cite{Raouf-19} that if $\nu$ is a probability
 measure on $\mathbb{R}$ with $m_0<+\infty$, then for $\alpha>0$, we have
\begin{equation}\label{VU+alpha}
V_{\nu^{\uplus \alpha}}(m)=\alpha V_{\nu}(m/\alpha)+m(m-\alpha m_0)(1/\alpha-1).
\end{equation}
Pursuing the study of the Cauchy-Stieltjes kernel families, we
determine in the present paper the effect on the variance function
of the free multiplicative convolution. We also use the variance
functions to re-derive in a easy way the limit theorems given in
\cite{Noriyoshi Sakuma-Hiroaki Yoshida} and involving the different
types of convolution. Similar results are established replacing the
free additive convolution by the boolean additive convolution. In
the rest of this section, we recall a few features about
Cauchy-Stieltjes kernel families. Our notations are the ones used in
\cite{Raouf-18}.\\Let $\nu$ be a non-degenerate probability measure
with support bounded from above. Then
\begin{equation}   \label{M(theta)}
M_{\nu}(\theta)=\int \frac{1}{1- \theta x} \nu(dx)
\end{equation}
 is defined for all $\theta\in [0,\theta_+)$ with $1/\theta_+=\max\{0, \sup {\rm supp}
 (\nu)\}$.\\
 For $\theta\in [0,\theta_+)$, we set
 $$P_{(\theta,\nu)}(dx)=\frac{1}{M_{\nu}(\theta)(1-\theta
x)}\nu(dx).$$ The set
 $$\mathcal{K}_+(\nu)=\{P_{(\theta,\nu)}(dx); \theta\in(0,\theta_+)\}$$
is called the one-sided Cauchy-Stieltjes kernel family generated by
$\nu$.\\ Let $k_{\nu}(\theta)=\int x P_{(\theta,\nu)}(dx)$ denote
the mean of $P_{(\theta,\nu)}$. Then the map $\theta\mapsto
k_{\nu}(\theta)$ is strictly increasing on $(0,\theta_+)$, it is
given by the formula
\begin{equation}\label{L2m}
k_{\nu}(\theta)=\frac{M_{\nu}(\theta)-1}{\theta M_{\nu}(\theta)}.
\end{equation}
The image of $(0,\theta_+)$ by $k_{\nu}$ is called the (one sided)
domain of the means of the family $\mathcal{K}_+(\nu)$, it is
denoted $(m_0(\nu),m_+(\nu))$. This leads to a parametrization of
the family $\mathcal{K}_+(\nu)$ by the mean. In fact, denoting by
$\psi_{\nu}$ the reciprocal of $k_{\nu}$, and writing for $m\in
(m_0(\nu),m_+(\nu))$, $Q_{(m,\nu)}(dx)=P_{(\psi_{\nu}(m),\nu)}(dx)$, we have
that
 $$\mathcal{K}_+(\nu)=\{Q_{(m,\nu)}(dx); m\in
(m_0(\nu),m_+(\nu))\}.$$ Now let
 \begin{equation}\label{B(nu)}
 B=B(\nu)=\max\{0,\sup {\rm supp}(\nu)\}=1/\theta_+\in[0,\infty).
 \end{equation}
 Then it is shown in \cite{Bryc-Hassairi-09} that the bounds $m_0(\nu)$ and $m_+(\nu)$ of the one-sided
 domain of means $(m_0(\nu),m_+(\nu))$ are given by
\begin{equation}\label{m0}
m_0(\nu)=\lim_{\theta\to 0^+} k_{\nu}(\theta)
\end{equation}
and with $B=B(\nu)$,
\begin{equation}\label{m+}
m_+(\nu)=B-\lim_{z\to B^+}\frac{1}{G_{\nu}(z)},
\end{equation}
where $G_{\nu}(z)$ is the Cauchy transform of $\nu$ given by
\begin{equation}\label{G-transform}
G_\nu(z)=\int\frac{1}{z-x}\nu(dx).
\end{equation}
%\end{proposition}
It is worth mentioning here that one may define the one-sided
Cauchy-Stieltjes kernel family for a measure $\nu$ with support
bounded from below. This family is usually denoted
${\mathcal{K}}_-(\nu)$ and parameterized by $\theta$ such that
$\theta_-<\theta<0$, where $\theta_-$ is either $1/b(\nu)$ or
$-\infty$ with $b=b(\nu)=\min\{0,\inf supp(\nu)\}$.
%$$\mathcal{K}_-(\nu)=\left\{P_{(\theta,\nu)}(dx)=\frac{1}{M_{\nu}(\theta)(1-\theta x)}\nu(dx),\ \ \theta\in(\theta_-,0)\right\}.$$
The domain of the means for ${\mathcal{K}}_-(\nu)$ is the interval
$(m_-(\nu),m_0(\nu))$ with $m_-(\nu)=b-1/G_\nu(b)$.
\\If $\nu$ has compact support, the natural domain for the parameter
$\theta$ of the two-sided \CSK family
$\mathcal{K}(\nu)=\mathcal{K}_+(\nu)\cup\mathcal{K}_-(\nu)\cup\{\nu\}$
is $\theta_-<\theta<\theta_+$. \\We come now to the notions of
variance and pseudo-variance functions. The variance function
\begin{equation}
  \label{Def Var}
  m\mapsto V_{\nu}(m)=\int (x-m)^2 Q_{(m,\nu)}(dx)
\end{equation}
 is a fundamental concept both in the theory of natural exponential
 families
 and in the theory of Cauchy-Stieltjes kernel families as presented
 in \cite{Bryc-06-08}. Unfortunately, if $\nu$ hasn't a first moment which is for example the case for a 1/2-stable law,
 all the distributions in the Cauchy-Stieltjes kernel family generated by $\nu$ have infinite variance.
 This fact has led the authors in \cite{Bryc-Hassairi-09} to introduce a notion of pseudo-variance
 function defined by
 \begin{equation}\label{m2v}
\frac{\V_{\nu}(m)}{m}=\frac{1}{\psi_{\nu}(m)}-m,
\end{equation}
If $m_0(\nu)=\int x d\nu$ is finite, then (see
\cite{Bryc-Hassairi-09}) the pseudo-variance
 function is related to the variance
 function by
\begin{equation}\label{VV2V}
\frac{\V_{\nu}(m)}{m}=\frac{V_{\nu}(m)}{m-m_0}.\end{equation}
 In particular, $\V_{\nu}=V_{\nu}$ when $m_0(\nu)=0$.
 The generating measure
$\nu$ is uniquely determined by the
 pseudo-variance function $\V_{\nu}$. In fact, if we set
\begin{equation}\label{z2m}
 z=z(m)=m+\frac{\V_{\nu}(m)}{ m},
 \end{equation}
 then the Cauchy transform satisfies
 \begin{equation}
  \label{G2V}
  G_\nu(z)=\frac{m}{\V_{\nu}(m)}.
\end{equation}
Also the distribution $Q_{(m,\nu)}(dx)$ may be written as $
Q_{(m,\nu)}(dx)=f_{\nu}(x,m)\nu(dx)$ with
\begin{equation} \label{F(V)}
f_{\nu}(x,m):=\left\{
               \begin{array}{ll}
                 \frac{\V_{\nu}(m)}{\V_{\nu}(m)+m(m-x)}, \ \ \ \ m\neq 0& \hbox{;} \\
                 1, \ \ \ \ \ \ \ \ \ \ \ \ \ \ \ \ \ \ \ \ m= 0, \ \ \V_{\nu}(0)\neq 0 & \hbox{;} \\
                 \frac{\V'_{\nu}(0)}{\V'_{\nu}(0)-x}, \ \ \ \ \ \ \ \ \ \ \ \ m= 0, \ \ \V_{\nu}(0)= 0 & \hbox{.}
               \end{array}
             \right.
\end{equation}
We now recall the effect on a Cauchy-Stieltjes kernel family of
applying an affine transformation to the generating measure.
Consider the affine transformation
 $$f:x\longmapsto(x-\lambda)/\beta$$ where $\beta\neq0$ and $\lambda\in\mathbb{R}$ and let $f(\nu)$ be the
 image of $\nu$ by $f$. In other words, if $X$ is a
 random variable with law $\nu$, then $f(\nu)$ is the law of
 $(X-\lambda)/\beta$, or
 $f(\nu)=D_{1/\beta}(\nu\boxplus\delta_{-\lambda})$, where
 $D_r(\mu)$ denotes the dilation of measure $\mu$ by a number
 $r\neq0$, that is $D_r(\mu)(U)=\mu(U/r)$.
 The point $m_0$ is transformed to $(m_0-\lambda)/\beta$.
 In particular, if $\beta<0$ the support of the measure $f(\nu)$ is bounded
 from below so that it generates the left-sided family $\mathcal{K}_-(f(\nu))$.
For $m$ close enough to $(m_0-\lambda)/\beta$, the pseudo-variance
function is
 \begin{equation}\label{pseudo-var affine transf}
 \mathbb{V}_{f(\nu)}(m)=\frac{m}{\beta(m\beta+\lambda)}\mathbb{V}_\nu(\beta m+\lambda).
 \end{equation}
 In particular, if the variance function exists, then
 $$V_{f(\nu)}(m)=\frac{1}{\beta^2}V_\nu(\beta m+\lambda).$$
 Note that using the special case where $f$ is the reflection $f(x)=-x$, on can transform
 a right-sided Cauchy-Stieltjes kernel family to a left-sided family. If
 $\nu$ has support bounded from above and its right-sided \CSK
 family $\mathcal{K}_+(\nu)$ has domain of means $(m_0,m_+)$ and
 pseudo-variance function $\mathbb{V}_\nu(m)$, then $f(\nu)$
 generates the left-sided \CSK family $\mathcal{K}_-(f(\nu))$
 with domain of means $(-m_+,-m_0)$ and pseudo-variance
 function $\mathbb{V}_{f(\nu)}(m)=\mathbb{V}_\nu(-m)$.

 To close this section, we state the following result due to Bryc \cite{Bryc-06-08}
 which is crucial in our method of proof of the limit theorems that will be given in Section 3.
\begin{proposition}\label{proposition2}
Let $V_{\nu_n}$ be a family of analytic functions which are variance
functions of a sequence of \CSK families
$\left(\mathcal{K}(\nu_n)\right)_{ n\geq1}$. \\If
$V_{\nu_n}\xrightarrow{n\to +\infty} V$ uniformly in a (complex)
neighborhood of $m_0\in\mathbb{R}$ and if $V(m_0)>0$, then there is
$\varepsilon>0$ such that $V$ is the variance function of a \CSK
family ${\mathcal{K}}(\nu)$, generated by a probability measure
$\nu$
parameterized by the mean $m\in(m_0-\varepsilon,m_0+\varepsilon)$.\\
Moreover, if a sequence of measures $\mu_n\in {\mathcal{K}}(\nu_n)$
such that $m_1=\int x \mu_n(dx)\in
(m_0-\varepsilon,m_0+\varepsilon)$ does not depends on $n$, then
$\mu_n\xrightarrow{n\to +\infty} \mu$ in distribution, where
$\mu\in{\mathcal{K}}(\nu)$ has the same mean $\int x \mu(dx)=m_1$.
\end{proposition}

\section{Free multiplicative convolution}

Let $\nu$ be a probability measure with support in
$\mathbb{R}_+=[0,+\infty)$ such that $\delta=\nu(\{0\})<1$, and
consider the function
\begin{equation}\label{Psi transform}
\Psi_{\nu}(z)=\int_0^{+\infty}\frac{zx}{1-zx}\nu(dx), \ \ \ \ \ z\in\mathbb{C}\setminus \mathbb{R}_+
\end{equation}
Denoting $\mathbb{C}_+=\{x+iy\in \mathbb{C}; \ y> 0\}$, the
function  $\Psi_{\nu}$ is univalent in the left half-plane
$i\mathbb{C}^+$ and its image $\Psi_{\nu}(i\mathbb{C}^+)$ is
contained in the disc with diameter $(\nu(\{0\})-1,0)$. Moreover
$\Psi_{\nu}(i\mathbb{C}^+)\cap\mathbb{R}=(\nu(\{0\})-1,0)$. Let
$\chi_{\nu}:\Psi_{\nu}(i\mathbb{C}^+)\longrightarrow i\mathbb{C}^+$
be the inverse function of $\Psi_{\nu}$. Then the $S-$transform of
$\nu$ is the function
\begin{equation}\label{Stransform}
S_{\nu}(z)=\chi_{\nu}(z)\frac{1+z}{z},
\end{equation}
and the $\Sigma$-transform of $\nu$ is given by
$$\Sigma_{\nu}(z)=S_{\nu}\left(\frac{z}{1-z}\right),\ \ \ \ \frac{z}{1-z}\in\Psi_{\nu}(i\mathbb{C}_+).$$
The product of S-transforms is an S-transform, and the
multiplicative free convolution $\nu_1\boxtimes \nu_2$ of the
measures $\nu_1$ and $ \nu_2$ is defined by $$S_{\nu_1\boxtimes
\nu_2}(z)=S_{\nu_1}(z)S_{ \nu_2}(z).$$ The $\Sigma$-transform also
satisfies $\Sigma_{\nu_1\boxtimes
\nu_2}(z)=\Sigma_{\nu_1}(z)\Sigma_{ \nu_2}(z)$.\\
Denoting by $\mathcal{M}$ (respectively by $\mathcal{M}_+$ ) the
space of Borel probability measures on $\mathbb{R}$ (respectively on
$\mathbb{R}_+$), we say that the probability measure $\nu \in
\mathcal{M}_+$ is infinitely divisible with respect to $\boxtimes$
if for each $n\in\mathbb{N}$, there exists $\nu_n\in \mathcal{M}_+$
such that
$$\nu=\underbrace{\nu_n \boxtimes.....\boxtimes \nu_n}_{ n \  \mbox{times}}.$$
The multiplicative free convolution power $\nu^{\boxtimes\alpha}$ is
defined at least for $\alpha\geq 1$
%(see \cite[Theorem 2.17]{Serban Teodor Belinschi})
by $S_{\nu^{\boxtimes\alpha}}(z)=S_{\nu}(z)^{\alpha}$. For more
details about the S-transform, see \cite{Bercovici-Voiculescu-93}.\\
The following technical result will be used in the proof of the
relation between the $S$-transform and the pseudo-variance function.

\begin{proposition}\label{proposition}
Let ${\mathcal{K}_{-}}(\nu)$ be the \CSK family generated by a non
degenerate probability measure $\nu$ concentrated on the positive
real line such that $\delta=\nu(\{0\})<1$. Let $m_0(\nu)$ be the
mean of $\nu$, and let $\V_{\nu}$ be the pseudo-variance function of
${\mathcal{K}_{-}}(\nu)$. If we set
$$\widetilde{z}=\widetilde{z}(m)=\psi_{\nu}(m)=\frac{1}{m+\V_{\nu}(m)/m},$$
then we have
\begin{equation}\label{PsitransformV}
\Psi_{\nu}(\widetilde{z})=\displaystyle\frac{m^2}{\V_{\nu}(m)}.
\end{equation}
Moreover  $\delta-1<\frac{m^2}{\V_{\nu}(m)}<0$, for all $m\in
(m_-(\nu),m_0(\nu))$.
\end{proposition}
\begin{proof}We have that
$\Psi_{\nu}(\widetilde{z})=\frac{1}{\widetilde{z}}G_{\nu}\left(\frac{1}{\widetilde{z}}\right)-1$.
With $\widetilde{z}=\psi_{\nu}(m)=\frac{1}{m+\V_{\nu}(m)/m}$, using
\eqref{G2V}, we have that
$$\Psi_{\nu}(\psi_{\nu}(m))=\frac{1}{\psi_{\nu}(m)}G_{\nu}\left(\frac{1}{\psi_{\nu}(m)}\right)-1
=(m+\V_{\nu}(m)/m)G_{\nu}(m+\V_{\nu}(m)/m)-1
=\frac{m^2}{\V_{\nu}(m)}.$$ On the other hand, from the fact that
$\Psi_{\nu}(i\mathbb{C}_+)\cap\mathbb{R}=(\nu(\{0\})-1,0)$ (see
\cite{Bercovici-Voiculescu-93} ), Equation\eqref{PsitransformV}
implies that $\delta-1<\frac{m^2}{\V_{\nu}(m)}<0$.
\end{proof}
The following result lists some useful properties of the
$S$-transform.
\begin{proposition}\label{prop}
Let ${\mathcal{K}_{-}}(\nu)$ be a one sided \CSK family generated by
a probability measure $\nu$ concentrated on the positive real line
such that $\delta=\nu(\{0\})<1$. Then we have
 \begin{itemize}
 \item[(i)] $S_{\nu}(z)>0$ for $z\in(\delta-1,0)$, and $S_{\nu}$ is strictly decreasing on $(\delta-1,0)$.
 \item[(ii)] $S_{\nu}(\delta-1,0)=(1/m_0(\nu),1/m_-(\nu))$.
 \item[(iii)] For $m\in (m_-(\nu),m_0(\nu))$
\begin{equation}\label{StransformV}
S_{\nu}\left(\frac{m^2}{\V_{\nu}(m)}\right)=\frac{1}{m}.
\end{equation}
\item[(iv)] $\displaystyle\lim_{z\longrightarrow 0}zS_{\nu}(z)=0.$
\end{itemize}
\end{proposition}

\begin{proof}
(i)This is proved in \cite[Proposition
6.8]{Bercovici-Voiculescu-93}, see also  \cite[Lemma 2]{Haagerup-Moller}.

 (ii) It is shown in
\cite[Lemma 4]{Haagerup-Moller} that
$S_{\nu}(\delta-1,0)=(b^{-1}, a^{-1})$, where
$$a=\left(\int_0^{+\infty}x^{-1}\nu(dx)\right)^{-1}, \ \ \ \textrm{and} \ \ \  b =\int_0^{+\infty}x\nu(dx).$$
On the other hand, from \cite[Corollary 2.3]{Raouf-17}, we have that
$$a^{-1}=\int_0^{+\infty}x^{-1}\nu(dx)=\widetilde{m}_0=-G_\nu(0)=1/m_-(\nu).$$ This
implies that
$$S_{\nu}(\delta-1,0)=(b^{-1}, a^{-1})=(1/m_0(\nu),1/m_-(\nu)).$$
(iii) As $\chi_{\nu}$ is the inverse function of $\Psi_{\nu}$, then
\eqref{PsitransformV}
 implies that $\chi_{\nu}(m^2/\V_{\nu}(m))=\psi_{\nu}(m).$ Using \eqref{Stransform}, we get
$$S_{\nu}\left(\frac{m^2}{\V_{\nu}(m)}\right)= \chi_{\nu}(m^2/\V_{\nu}(m))\frac{1+m^2/\V_{\nu}(m)}{m^2/\V_{\nu}(m)}=1/m.$$

(iv) According to \cite[Corollary 3.6]{Bryc-Hassairi-09}, we have
that if $\V_{\nu}$ is the pseudo-variance function of a \CSK family
generated by probability measure with support bounded from one side
(say from above), then $$\frac{m}{\V_{\nu}(m)}\longrightarrow 0\ \
\textrm{and }\ \ \frac{m^2}{\V_{\nu}(m)}\longrightarrow 0 \ \
\textrm{as} \ m\longrightarrow m_0.$$ Using \eqref{StransformV} for
$z=m^2/\V_{\nu}(m)$, this implies that
$zS_{\nu}(z)=m/\V_{\nu}(m)\longrightarrow 0$ as $m\longrightarrow
m_0$.
\end{proof}
As a consequence of the previous proposition, we deduce the
following property of the pseudo-variance function $\V$.
\begin{corollary}\label{corollary}
The function $m\longmapsto m^2/\V_{\nu}(m)$ is analytic and strictly
increasing on the domain of the mean $(m_-(\nu),m_0(\nu))$.
\end{corollary}
\begin{proof}
From \eqref{StransformV}, we have that
$m^2/\V_{\nu}(m)=\chi_{\nu}^{-1}(\psi_{\nu}(m))$. We already know
that $\psi_{\nu}$ is analytic and strictly increasing, and as
$\chi_{\nu}$ is also analytic and strictly increasing on
$(\delta-1,0)$ (see\cite{Bercovici-Voiculescu-93}), we deduce that
the function $m\longmapsto m^2/\V_{\nu}(m)$ is analytic and strictly
increasing $(m_-(\nu),m_0(\nu))$.
\end{proof}
 We now state and prove our main result concerning the
effect of the free multiplicative convolution on a \CSK family.
\begin{theorem}\label{TH}
Let $\V_{\nu}$ be the pseudo-variance function of the \CSK family
${\mathcal{K}_{-}}(\nu)$ generated by a non degenerate  probability
distribution $\nu$ concentrated on the positive real line with mean
$m_0(\nu)$. Consider $\alpha>0$ such that $\nu^{\boxtimes \alpha}$
is defined. Then
 \begin{itemize}
\item[(i)] $m_-(\nu^{\boxtimes \alpha})=(m_-(\nu))^{\alpha}$ and $m_0(\nu^{\boxtimes
\alpha})=(m_0(\nu))^{\alpha}$, and for $m\in(m_-(\nu^{\boxtimes
\alpha}),m_0(\nu^{\boxtimes \alpha})),$
\begin{equation}\label{pseudovarnualpha}
\V_{\nu^{\boxtimes \alpha}}(m)=m^{2-2/\alpha}\V_{\nu}\left(m^{1/\alpha}\right).
\end{equation}
\item[(ii)] If $m_0<+\infty$, then the variance functions of the \CSK
families generated by $\nu$ and $\nu^{\boxtimes \alpha}$ exist and
\begin{equation}\label{varnualpha}
V_{\nu^{\boxtimes \alpha}}(m)=\frac{m-m_0^{\alpha}}{m^{1/\alpha}-m_0}m^{1-1/\alpha}V_{\nu}\left(m^{1/\alpha}\right).
\end{equation}
 \end{itemize}
\end{theorem}
\begin{proof}
(i) The fact that $m_0(\nu^{\boxtimes \alpha})=(m_0(\nu))^{\alpha}$
and $m_-(\nu^{\boxtimes \alpha})=(m_-(\nu))^{\alpha}$ comes from
Proposition \ref{prop} and from the multiplicative property of the
S-transform. On the the other hand, for
$m\in((m_-(\nu))^{\alpha},(m_0(\nu))^{\alpha})$, we have that
$m^{1/\alpha}\in(m_-(\nu),m_0(\nu))$ and $m^2/\V_{\nu^{\boxtimes
\alpha}}(m)\in(\delta-1,0)$. We apply \eqref{StransformV} and the
multiplicative property of the S-transform to see that
 $$S_{\nu}\left(\frac{m^2}{\V_{\nu^{\boxtimes \alpha}}(m)}\right)=\left[S_{\nu^{\boxtimes \alpha}}\left(\frac{m^2}{\V_{\nu^{\boxtimes \alpha}}(m)}\right)\right]^{1/\alpha}=\frac{1}{m^{1/\alpha}}=S_{\nu}\left(\frac{m^{2/\alpha}}{\V_{\nu}(m^{1/\alpha})}\right).$$
 This implies that
 $$\frac{m^2}{\V_{\nu^{\boxtimes \alpha}}(m)}=\frac{m^{2/\alpha}}{\V_{\nu}(m^{1/\alpha})}.$$

 (ii) If $m_0<+\infty$, then the variance functions of the \CSK families
  ${\mathcal{K}_{-}}(\nu)$ and ${\mathcal{K}_{-}}(\nu^{\boxtimes \alpha})$ exist and the relation \eqref{varnualpha} follows from \eqref{VV2V} and \eqref{pseudovarnualpha}.
 %\begin{eqnarray*}
% \mathbf{M}_+(\nu^{\boxtimes \alpha})
% & = & \inf\left\{m>m_0(\nu^{\boxtimes \alpha}) : \frac{\V_{\nu^{\boxtimes \alpha}}(m)}{m}<0\right\}\\
% & = & \inf\left\{m>(m_0(\nu))^{\alpha} : m^{1-1/\alpha}\frac{\V_{\nu}(m^{1/\alpha})}{m^{1/\alpha}}<0\right\}\\
% & = & \inf\left\{m^{1/\alpha}>m_0(\nu) : \frac{\V_{\nu}(m^{1/\alpha})}{m^{1/\alpha}}<0\right\}=(\mathbf{M}_+(\nu))^{\alpha}.
% \end{eqnarray*}
% It is clear that $m^{1-1/\alpha}$ is positive.
\end{proof}

\section{Limit theorems.}

Several limit theorems involving the free additive convolution
$\boxplus$, the Boolean additive convolution $\uplus$ and the free
multiplicative convolution $\boxtimes$ have been established in
\cite{W. Mlotkowski} and in \cite{Noriyoshi Sakuma-Hiroaki Yoshida}.
In this section, we use the variance functions to re-derive these
results. According to Proposition \ref{proposition2}, this leads to
some new variance functions with non usual form. Recall that the
definition of the free additive convolution is based on the notion
of ${\mathcal{R}}$-transform (see \cite{Bercovici-Voiculescu-93}).
In fact, it is proved (\cite{Bercovici-Voiculescu-93}) that the
inverse $G_\nu^{-1}$ of $G_{\nu}$ is defined on a domain of the form
 $$\{z\in\mathbb{C}\ :\ \Re z> c,\ | z|< M\},$$
where $c$ and $M$ are two positive constants. \\The
${\mathcal{R}}$-transform is defined in the same domain by
\begin{equation}
{\mathcal{R}}_\nu(z)=G_\nu^{-1}(z)-1/z,
\end{equation}
and the free additive convolution  $\mu\boxplus\nu$ of the
probability measures $\mu$, $\nu$ on Borel sets of the real line is
a uniquely defined probability measure $\mu\boxplus\nu$ such that
\begin{equation}\label{property lin R-trans}
{\mathcal{R}}_{\mu\boxplus\nu}(z)={\mathcal{R}}_{\mu}(z)+{\mathcal{R}}_{\nu}(z)
\end{equation}
for all $z$ in an appropriate domain (see \cite[Sect.
5]{Bercovici-Voiculescu-93} for details).\\ A probability measure
$\nu \in \mathcal{M}$ is $\boxplus-$infinitely divisible, if for
each $n\in\mathbb{N}$, there exists $\nu_n\in \mathcal{M}$ such that
$$\nu=\underbrace{\nu_n \boxplus.....\boxplus \nu_n}_{ n \  \mbox{times}}.$$
Concerning the Boolean additive convolution, its definition uses the
notion of $K$-transform (see \cite{Speicher-Woroudi-93}). If $\nu$
is a probability measure on $\mathbb{R}$, then the $K$-transform of
$\nu$ is given by
\begin{equation}\label{Knu}
K_{\nu}(z)=z-\frac{1}{G_{\nu}(z)}, \ \ \ \ for \ z\in \mathbb{C}^+.
\end{equation}
The function $K_{\nu}$ is usually called self energy, it represents
the analytic backbone of the boolean additive convolution. \\ For
two probability measures $\mu$ and $\nu$, the additive Boolean
convolution $\mu\uplus\nu$ is the probability measure defined by
\begin{equation}\label{Kuplus}
K_{\mu\uplus\nu}(z)=K_{\mu}(z)+K_{\nu}(z), \ \ \ \ \mbox{for}\ \ \ z\in
\mathbb{C}^+.
\end{equation}
A probability measure $\nu \in \mathcal{M}$ is infinitely divisible
in the boolean sense if for each $n\in\mathbb{N}$, there exists
$\nu_n\in \mathcal{M}$ such that
$$\nu=\underbrace{\nu_n \uplus.....\uplus \nu_n}_{ n \  \mbox{times}}.$$

We are now in position to state and prove the main result of the
section.
\begin{theorem}\label{TH3}
Let $\nu$ be in $\mathcal{M}_+$ with mean $m_0(\nu)>0$. Suppose that
$\nu$ has a finite second moment. Then denoting
$\gamma=\frac{Var(\nu)}{(m_0(\nu))^2}=\frac{V_{\nu}(m_0)}{m_0^2}$,
we have
\begin{itemize}

 \item[(i)]
 $$D_{1/(nm_0^n)}\left(\nu^{\boxtimes n}\right)^{\boxplus n} \xrightarrow{n\to +\infty} \eta_{\gamma} \ \
   \ \ \ \  \ \mbox{in distribution},$$
where $\eta_{\gamma}$ is such that $m_0(\eta_{\gamma})=1$,
$(m_-(\eta_{\gamma}),m_0(\eta_{\gamma}))\subset (0,1)$ and
 the variance function of the \CSK family generated by $\eta_{\gamma}$ is given for $m\in (m_-(\eta_{\gamma}),m_0(\eta_{\gamma}))$,by
\begin{equation}\label{V-eta-gamma}
 V_{\eta_{\gamma}}(m)=\frac{m(m-1)}{m_0^2\ln(m)}V_{\nu}(m_0)=\frac{\gamma m(m-1)}{\ln(m)}.
 \end{equation}
 \item[(ii)]  $$D_{1/(nm_0^n)}\left(\nu^{\boxtimes n}\right)^{\uplus n}\xrightarrow{n\to +\infty} \sigma_{\gamma} \ \ \ \ \ \ \ \ \mbox{in distribution},$$
where $\sigma_{\gamma}$ is such that
$m_0(\sigma_{\gamma})=1$,
$(m_-(\sigma_{\gamma}),m_0(\sigma_{\gamma}))\subset
(0,1)$, and
 for all  $m\in (m_-(\sigma_{\gamma}),m_0(\sigma_{\gamma}))$, the variance function of the \CSK family generated by $\sigma_{\gamma}$
is given by
 \begin{equation}\label{V-sigma-gamma}
V_{\sigma_{\gamma}}(m)=\frac{m(m-1)}{m_0^2 \ln(m)}V_{\nu}(m_0)+m(1-m)=\frac{\gamma m(m-1)}{ \ln(m)}+m(1-m)
  \end{equation}
\end{itemize}
\end{theorem}
\begin{proof}
 (i)
We have that
$$
m_0\left(D_{1/(nm_0(\nu)^n)}\left(\nu^{\boxtimes n}\right)^{\boxplus n}\right)
=  \frac{1}{(nm_0(\nu)^n)}m_0\left(\left(\nu^{\boxtimes n}\right)^{\boxplus n}\right)=\frac{1}{(m_0(\nu))^n}m_0\left(\nu^{\boxtimes n}\right)=1.
$$
Furthermore,
\begin{eqnarray*}
V_{D_{\frac{1}{nm_0^n}}((\nu^{\boxtimes n})^{\boxplus n})}(m)
& = & \frac{1}{n^2m_0^{2n}}V_{(\nu^{\boxtimes n})^{\boxplus n}}(nmm_0^n)\\
& = & \frac{1}{nm_0^{2n}}V_{\nu^{\boxtimes n}}(mm_0^n)\\
& = &\frac{m^{1-1/n}}{nm_0^2}\frac{(m-1)}{m^{1/n}-1}V_{\nu}\left(m_0m^{1/n}\right)\\
& = & \frac{m^{1-1/n}(m-1)}{m_0^2\frac{m^{1/n}-1}{1/n}}V_{\nu}\left(m_0m^{1/n}\right)\\
& \xrightarrow{n\to +\infty} & \frac{m(m-1)}{m_0^2\left(\exp(x\ln(m))'|_{x=0}\right)}V_{\nu}(m_0)\\
& = & \frac{m(m-1)}{m_0^2\ln(m)}V_{\nu}(m_0).
\end{eqnarray*}
According to Proposition \ref{proposition2}, this implies that
$$D_{1/(nm_0^n)}\left(\nu^{\boxtimes n}\right)^{\boxplus n} \xrightarrow{n\to +\infty} \eta_{\gamma} \ \
  \ \ \ \  \ \mbox{in distribution},$$
  where $$V_{\eta_{\gamma}}(m)=\frac{m(m-1)}{m_0^2\ln(m)}V_{\nu}(m_0)=\frac{\gamma m(m-1)}{\ln(m)},$$
  and $m_0(\eta_{\gamma})=m_0\left(D_{1/(nm_0(\nu)^n)}\left(\nu^{\boxtimes n}\right)^{\boxplus
  n}\right)=1$.\\
On the other hand, it is well known that if a sequence of
probability measures $\mu_n$ in $\mathcal{M}_+$ is such that
$\mu_n\xrightarrow{n\to +\infty} \mu$ in distribution, then $\mu\in
\mathcal{M}_+$. \\Therefore $\eta_{\gamma} \in\mathcal{M}_+$ and
$m=\int_0^{+\infty} x Q_{(m,\eta_{\gamma})}(dx)>0$. \\This with the
fact that $m_0(\eta_{\gamma})=1$ implies that
$(m_-(\eta_{\gamma}),m_0(\eta_{\gamma}))\subset (0,1)$.

(ii)

 From the fact that $m_0(\nu^{\uplus n})=nm_0(\nu)$ (see
\cite[Theorem 2.3]{Raouf-19}), we have that
$$
m_0\left(D_{1/(nm_0(\nu)^n)}\left(\nu^{\boxtimes n}\right)^{\uplus n}\right)
=  \frac{1}{(nm_0(\nu)^n)}m_0\left(\left(\nu^{\boxtimes n}\right)^{\uplus n}\right)=\frac{1}{(m_0(\nu))^n}m_0\left(\nu^{\boxtimes n}\right)=1.
$$
Furthermore,
\begin{eqnarray*}
V_{D_{\frac{1}{nm_0^n}}((\nu^{\boxtimes n})^{\uplus n})}(m)
 & = &  \frac{1}{n^2m_0^{2n}}V_{(\nu^{\boxtimes n})^{\uplus n}}(nmm_0^n)\\
 & = &  \frac{1}{nm_0^{2n}}V_{\nu^{\boxtimes n}}(mm_0^n)+m(1-m)(1-1/n)\\
 & = &  \frac{(mm_0^n-m_0^n)m^{1-1/n}m_0^{n-1}V_{\nu}\left(m^{1/n}m_0\right)}{nm_0^{2n}[(mm_0^n)^{1/n}-m_0]}
 + m \left(1-m\right)\left(1-1/n\right).\\
 & = & \frac{(m-1)m^{1-1/n}}{m_0^2\frac{m^{1/n}-1}{1/n}}
 V_{\nu}\left(m^{1/n}m_0\right)+  m\left(1-m\right)\left(1-1/n\right).\\
 & \xrightarrow{n\to +\infty} & \frac{m(m-1)}{m_0^2 \ln(m)}V_{\nu}(m_0)+m(1-m).
\end{eqnarray*}
According to Proposition \ref{proposition2}, this implies that
$$D_{1/(nm_0^n)}\left(\nu^{\boxtimes n}\right)^{\uplus n} \xrightarrow{n\to +\infty} \sigma_{\gamma} \ \
  \ \ \ \  \ \mbox{in distribution},$$
  where $$V_{\sigma_{\gamma}}(m)=\frac{m(m-1)}{m_0^2 \ln(m)}V_{\nu}(m_0)+m(1-m)=\frac{\gamma m(m-1)}{ \ln(m)}+m(1-m),$$
  and $m_0(\sigma_{\gamma})=m_0\left(D_{1/(nm_0(\nu)^n)}\left(\nu^{\boxtimes n}\right)^{\uplus
  n}\right)=1$.\\
Now the fact that $\sigma_{\gamma} \in\mathcal{M}_+$ implies that
 $m=\int_0^{+\infty} x Q_{(m,\sigma_{\gamma})}(dx)>0$, and
given that $m_0(\sigma_{\gamma})=1$, we deduce that
 $(m_-(\sigma_{\gamma}),m_0(\sigma_{\gamma}))\subset (0,1)$.
\end{proof}
It is worth mentioning here that according to \cite{Noriyoshi
Sakuma-Hiroaki Yoshida}, the S-transform of the limit distribution
$\eta_{\gamma}$ is given by
$$S_{\eta_{\gamma}}(z)=\exp(-\gamma z).$$
with $m_0(\eta_{\gamma})=1/S_{\eta_{\gamma}}(0)=1$. For
$z=m^2/\V_{\eta_{\gamma}}(m)$, \eqref{StransformV} becomes
$$1/m=S_{\eta_{\gamma}}\left(m^2/\V_{\eta_{\gamma}}(m)\right)=\exp\left(-\gamma m^2/\V_{\eta_{\gamma}}(m)\right).$$
This implies that $$\V_{\eta_{\gamma}}(m)=\frac{\gamma m^2}{\ln
(m)},$$ and consequently, the variance function of the \CSK family
generated by $\eta_{\gamma}$ is given by
$$V_{\eta_{\gamma}}(m)=\frac{m-m_0(\eta_{\gamma})}{m}\V_{\eta_{\gamma}}(m)=\frac{m-m_0(\eta_{\gamma})}{m}\frac{\gamma m^2}{\ln (m)}=\frac{\gamma m(m-1)}{\ln(m)}$$
which is nothing but \eqref{V-eta-gamma}.

Also the $\Sigma$-transform of the limit distribution
$\sigma_{\gamma}$ is given in \cite{Noriyoshi Sakuma-Hiroaki
Yoshida} by
$$\Sigma_{\sigma_{\gamma}}(z)=\exp(-\gamma z),$$
with
$m_0(\sigma_{\gamma})=1/S_{\sigma_{\gamma}}(0)=1/\Sigma_{\sigma_{\gamma}}(0)=1$.
Setting $z=\frac{m^2}{\V_{\sigma_{\gamma}}(m)+m^2}$, we get
$$\exp\left(-\gamma \frac{m^2}{\V_{\sigma_{\gamma}}(m)+m^2}\right)=\Sigma_{\sigma_{\gamma}}\left
(\frac{m^2}{\V_{\sigma_{\gamma}}(m)+m^2}\right)=S_{\sigma_{\gamma}}(m^2/\V_{\sigma_{\gamma}}(m))=1/m.$$
This implies that $ \frac{\gamma
m^2}{\V_{\sigma_{\gamma}}(m)+m^2}=\ln(m)$, that is
$$\V_{\sigma_{\gamma}}(m)=\frac{\gamma m^2}{\ln(m)}-m^2.$$
The variance function of the \CSK family generated by
$\sigma_{\gamma}$ is then given by
$$V_{\sigma_{\gamma}}(m)=\frac{m-m_0(\sigma_{\gamma})}{m}\V_{\sigma_{\gamma}}(m)
=\frac{m-1}{m}\left(\frac{\gamma m^2}{\ln(m)}-m^2\right),$$ which is
\eqref{V-sigma-gamma}.
\begin{corollary}\label{cor1}
For the free Poisson law $\mu(dx)=\frac{1}{2\pi}\sqrt{\frac{4-x}{x}}\textbf{1}_{(0,4)}(x)dx$, we have
$$D_{1/n}\left(\mu^{\boxtimes n}\right)^{\uplus n}\xrightarrow{n\to +\infty} \sigma_1  \ \ \ \ \mbox{in distribution},$$
with  $m_0(\sigma_1)=1$, and for all $m\in
(m_-(\sigma_1),m_0(\sigma_1))$  the variance function of the \CSK
family generated by $\sigma_1$ is
$$V_{\sigma_1}(m)=\frac{ m(m-1)}{ \ln(m)}+m(1-m).$$
\end{corollary}
\begin{proof}
We have that for $a$ such that $0<a^2\leq1$, the absolutely
continuous centered Marchenko-Pastur  distribution
 $$\nu(dx)=\displaystyle\frac{\sqrt{4-(x-a)^2}}{2\pi(1+ax)}\textbf{1}_{(a-2,a+2)}(x)dx$$
 generates the  \CSK family with variance function
$V_{\nu}(m)=1+am={\V}_{\nu}(m)$ and two sided domain of means
$(-1,1)$. The probability measure $\mu$ is the image of $\nu$, for
$a=1$, by the map $x\longmapsto 1+x$. It generates the \CSK family
${\mathcal{K}}(\mu)$ with variance function $V_{\mu}(m)=m$ and
two-sided domain of means $(0,2)$ with $m_0(\mu)=1$. \\ From Theorem
\ref{TH3} (ii), we have that
$$D_{1/n}\left(\mu^{\boxtimes n}\right)^{\uplus n}\xrightarrow{n\to +\infty} \sigma_1 \ \ \ \ \ \ \ \mbox{in distribution},$$
with $m_0(\sigma_1)=m_0\left(D_{1/n}\left(\mu^{\boxtimes
n}\right)^{\uplus n}\right) =1$ and  for all $m\in
(m_-(\sigma_1),m_0(\sigma_1))\subset (0,1)$  the variance function
is
$$V_{\sigma_1}(m)=\frac{ m(m-1)}{ \ln(m)}+m(1-m).$$
\end{proof}
In what follows, we give the link between the two limit probability
measures $\eta_{\gamma}$ and $\sigma_{\gamma}$ by mean of the
boolean Bercovici-Pata transformation. For instance, Belinschi and
Nica \cite{Belinschi-Nica-08} have defined for $t\geq0$, the mapping
\begin{eqnarray*}
\mathbb{B}_t:\mathcal{M} &\rightarrow& \mathcal{M}\\
 \mu& \mapsto&
\left(\mu^{\boxplus (1+t)}\right)^{\uplus \frac{1}{1+t}}.
\end{eqnarray*}
 They have also proved that $\mathbb{B}_1$ coincides with the canonical bijection
from $\mathcal{M}$ into $\mathcal{M}_{Inf-div}$ discovered by
Bercovici and Pata in their study of the relations between infinite
divisibility in free and in Boolean probability. Here
$\mathcal{M}_{Inf-div}$ stands for the set of probability
distributions in $\mathcal{M}$ which are infinitely divisible with
respect to the operation $\boxplus$. From \cite{Belinschi-Nica-08},
we have that for $t, \ s\geq0$,  $\mathbb{B}_t\circ
\mathbb{B}_s=\mathbb{B}_{t+s}$. This implies that for $t\geq1$,
$\mathbb{B}_1(\mathbb{B}_{t-1}(\mathcal{M}))\subseteq
\mathcal{M}_{Inf-div}$ (see \cite[Corollary
3.1]{Belinschi-Nica-08}), and consequently $\mathbb{B}_t(\mu)$ is
$\boxplus$-infinitely divisible for $\mu\in \mathcal{M}$ and $t \geq
1$. On the other hand, the pseudo-variance function of the \CSK
family generated by $\mathbb{B}_t(\mu)$ is given in \cite{Raouf-19}.
More precisely, it is shown that if $\mu$ is a probability measure
on the real line with support bounded from above, then for
$m>m_0(\mu)$ close enough to $m_0(\mu)$,
 \begin{equation}\label{VBtmu}
\V_{\mathbb{B}_t(\mu)}(m)=\V_{\mu}(m)+tm^2.
\end{equation}
If $m_0(\mu)<+\infty$, the variance functions of the \CSK families
generated by $\mu$ and $\mathbb{B}_t(\mu)$ exist and
\begin{equation}\label{vBtmu}
V_{\mathbb{B}_t(\mu)}(m)=V_{\mu}(m)+tm(m-m_0(\mu)).
\end{equation}
\begin{proposition}
Let $\nu$ be in $\mathcal{M}_+$ with mean $m_0(\nu)$. Suppose that
$\nu$ has a finite second moment. Then denoting
$\gamma=\frac{Var(\nu)}{(m_0(\nu))^2}=\frac{V_{\nu}(m_0)}{m_0^2}$,
we have
$$\eta_{\gamma}=\mathbb{B}_1(\sigma_{\gamma})$$
\end{proposition}
\begin{proof}
We have that $m_0(\mathbb{B}_1(\sigma_{\gamma}))=m_0(\sigma_{\gamma})=1=m_0(\eta_{\gamma})$ and for $m<m_0(\mathbb{B}_1(\sigma_{\gamma}))=m_0(\eta_{\gamma})=1$ close enough to $1$ we have
$$V_{\mathbb{B}_1(\sigma_{\gamma})}(m)=V_{\sigma_{\gamma}}(m)+m(m-1)=\frac{ \gamma m(m-1)}{ \ln(m)}+m(1-m)+m(m-1)=\frac{ \gamma m(m-1)}{ \ln(m)}=V_{\eta_{\gamma}}(m).$$
Since the variance function of a \CSK family together with the first
moment determine the corresponding generating probability measure,
we deduce that $\eta_{\gamma}=\mathbb{B}_1(\sigma_{\gamma})$.
\end{proof}

\end{document}